\documentclass[12pt, twoside, leqno]{article}
\usepackage{amsmath,amsthm}
\usepackage{amssymb,latexsym}
\usepackage{enumerate}

\newtheorem{thm}{Theorem}[section]
\newtheorem{cor}[thm]{Corollary}
\newtheorem{lem}[thm]{Lemma}

\newtheorem{prop}[thm]{Proposition}

\theoremstyle{definition}

\numberwithin{equation}{section}

\newcommand{\CQFD}
{%
\mbox{}%
\nolinebreak%
\hfill%
$\square$%
\medbreak%
\smallskip%
}
\def\ind{{\rm 1\kern-.27em I}}

\newcommand{\dis}{\displaystyle}
\newcommand{\eps}{\varepsilon}
\newcommand{\ud}{\mathrm{d}}
\newcommand{\D}{\mathbb{D}}
\newcommand{\C}{\mathbb{C}}

\newcommand{\N}{\mathbb{N}}
\newcommand{\T}{\mathbb{T}}

\begin{document}


\baselineskip=17pt


\title{Essential norms of weighted composition operators between Hardy spaces $H^p$ and $H^q$ for $1\leq p,q\leq \infty$}

\author{R. Demazeux\\
Univ Lille Nord de France, FRANCE\\
UArtois, Laboratoire de Math\'ematiques de Lens EA 2462,\\
F\'ed\'eration CNRS Nord-Pas-de-Calais FR 2956,\\
F-62 300 LENS, FRANCE\\
E-mail: romain.demazeux@euler.univ-artois.fr}

\date{}

\maketitle


\renewcommand{\thefootnote}{}



\renewcommand{\thefootnote}{\arabic{footnote}}
\setcounter{footnote}{0}


\begin{abstract}
We complete the different cases remaining in the estimation of the essential norm of a weighted composition operator acting  between the Hardy spaces $H^p$ and $H^q$ for~$1\leq p,q\leq\infty.$ In particular we give some estimates for the cases~$1=p\leq q\leq\infty$ and~$1\leq q<p\leq\infty.$
\end{abstract}

\section{Introduction}
Let $\D=\{z\in\C\mid\ |z|<1\}$ denote the open unit disk in the complex plane. Given two analytic functions~$u$ and~$\varphi$ defined on~$\D$ such that~$\varphi(\D)\subset\D$, one can define the \emph{weighted composition operator} $uC_\varphi$ that maps any analytic function~$f$ defined on~$\D$ into the function $uC_\varphi(f)=u(f\circ\varphi).$ In~\cite{H}, de Leeuw showed that the isometries in the Hardy space~$H^1$ are weighted composition operators, while Forelli~\cite{F} obtained this result for the Hardy space~$H^p$ when $1<p<\infty,\ p\neq 2$. Another example is the study of composition operators on the half-plane. A composition operator in a Hardy space of the half-plane is bounded if and only if a certain weighted composition operator is bounded on the Hardy space of the unit disk (see~\cite{M} and~\cite{SS}).\par

When $u\equiv 1,$ we just have the composition operator~$C_\varphi$. The continuity of these operators on the Hardy space $H^p$ is ensured by the Littlewood's subordination principle, which says that~$C_\varphi(f)$ belongs to $H^p$ whenever $f\in H^p$ ~(see \cite{CMC}, Corollary $2.24$). As a consequence, the condition \mbox{$u\in H^\infty$} suffices for the boundedness of~$uC_\varphi$ on~$H^p.$ Considering the image of the constant functions, a necessary condition is that~$u$ belongs to~$H^p.$ Nevertheless a weighted composition operator needs not to be continuous  on $H^p$, and it is easy to find examples where~$uC_\varphi(H^p)\nsubseteq H^p$~(see Lemma~$2.1$ of~\cite{CH} for instance). \par
In this note we deal with weighted composition operators between~$H^p$ and~$H^q$ for~\mbox{$1\leq p,q\leq\infty.$} Boundedness and compactness are characterized in~\cite{CH} for~$1\leq p\leq q<\infty$ by means of Carleson measures, while essential norms of weighted composition operators are estimated in~\cite{CZ1} for~$1<p\leq q<\infty$ by means of an integral operator. For the case~$1\leq q<p<\infty,$ boundedness and compactness of~$uC_\varphi$ are studied in~\cite{CZ1}, and Gorkin and MacCluer in~\cite{GM} gave an estimate of the essential norm of a composition operator acting between~$H^p$ and $H^q.$\par
The aim of this paper is to complete the different cases remaining in the estimation of the essential norm of a weighted composition operator. In section~$2$ and~$3,$ we give an estimate of the essential norm of~$uC_\varphi$ acting between $H^p$ and~$H^q$ when~$p=1$ and~$1\leq q<\infty$ and when~$1\leq p<\infty$ and~$q=\infty.$ Sections~$4$ and~$5$ are devoted to the case where~$\infty\geq p>q\geq 1.$\par
\vskip 0.5cm
Let $\overline{\D}$ be the closure of the unit disk~$\D$ and $\T=\partial\D$ its boundary. We denote by $\ud m=\ud t/2\pi$ the normalised Haar measure on~$\T.$ If $A$ is a Borel subset of~$\T$, the notation~$m(A)$ as well as~$|A|$ will design the Haar measure of~$A.$ For~$1\leq p<\infty$, the Hardy space $H^p(\D)$ is the space of analytic functions~$f:\D\to\C$ satisfying the following condition\[\|f\|_p=\sup_{0<r<1}\left(\int_\T|f(r\zeta)|^p\ \ud m(\zeta)\right)^{1/p}<\infty.\] Endowed with this norm, $H^p(\D)$ is a Banach space. The space~$H^\infty(\D)$ is consisting of every bounded analytic function on~$\D,$ and its norm is given by the supremum norm on~$\D$.\\
We recall that any function~$f\in H^p(\D)$ can be extended on~$\T$ to a function~$f^*$ by the following formula:~$f^*(e^{i\theta})=\lim_{r\nearrow 1}f(re^{i\theta})$. The limit exists almost everywhere by Fatou's theorem, and~$f^*\in L^p(\T)$. Moreover, $f\mapsto f^*$ is an into isometry from~$H^p(\D)$ to~$L^p(\T)$ whose image, denoted by~$H^p(\T)$ is the closure (weak-star closure for~$p=\infty$) of the set of polynomials in~$L^p(\T).$ So we can identify~$H^p(\D)$ and~$H^p(\T)$, and we will use the notation~$H^p$ for both of these spaces. More on Hardy spaces can be found in~\cite{K} for instance.\par

The \emph{essential norm} of an operator~$T:X\rightarrow Y,$ denoted~$\|T\|_e$, is given by\[\|T\|_e=\inf\{\|T-K\|\mid\ K\ \textrm{is a compact operator from }X\textrm{ to }Y\}.\]
Observe that $\|T\|_e\leq\|T\|,$ and $\|T\|_e$ is the norm of~$T$ seen as an element of the space~$B(X,Y)/K(X,Y)$ where $B(X,Y)$ is the space of all bounded operators from~$X$ to~$Y$ and $K(X,Y)$ is the subspace consisting of all compact operators.

Notation: we will write $a\approx b$ whenever there exists two positive universal constants~$c$ and~$C$ such that~$cb\leq a\leq Cb.$ In the sequel,~$u$ will be a \emph{non-zero analytic function} on~$\D$ and~$\varphi$ will be a \emph{non-constant analytic function} defined on~$\D$ satisfying~$\varphi(\D)\subset\D.$\\


\section{$uC_\varphi\in B(H^1,H^q)$ for $1\leq q<\infty$}

 Let us first start with a characterization of the boundedness of $uC_\varphi$ acting between~$H^p$ and $H^q$:

\begin{thm}[see {\cite[Theorem~$4$]{CZ1}}]\label{Hp-Hq}
Let $u$ be an analytic function on~$\D$ and $\varphi$ an analytic self-map of~$\D.$ Let~$0<p\leq q<\infty.$ Then the weighted composition operator $uC_\varphi$ is bounded from $H^p$ to~$H^q$ if and only if \[\sup_{a\in\D}\int_{\T}\vert u(\zeta)\vert^q\bigg(\frac{1-\vert a\vert^2}{\vert1-\bar{a}\varphi(\zeta)\vert^2}\bigg)^{q/p}\ \ud m(\zeta)<\infty.\]
\end{thm}

\noindent As a consequence~$uC_\varphi$ is a bounded operator as soon as~$uC_\varphi$ is uniformly bounded on the set~$\{k_a^{1/p}\mid\ a\in\D\}$ where~$k_a$ is the normalized kernel defined by~$k_a(z)=(1-\vert a\vert^2)/(1-\bar{a}z)^2,\ a\in\D.$
Note that~$k_a^{1/p}\in H^p$ and $\|k_a^{1/p}\|_p=1.$ These kernels play a crucial role in the estimation of the essential norm of a weighted composition operator:

 \begin{thm}[see {\cite[Theorem~$5$]{CZ1}}]\label{normeHp-Hq}
 Let $u$ be an analytic function on~$\D$ and $\varphi$ an analytic self-map of~$\D.$
 Assume that the weighted composition operator $uC_\varphi$ is bounded from $H^p$ to~$H^q$ with $1<p\leq q<\infty.$ Then\[\|uC_\varphi\|_e\approx\limsup_{\vert a\vert\to1^-}\Bigg(\int_{\T}\vert u(\zeta)\vert^q\bigg(\frac{1-\vert a\vert^2}{\vert1-\bar{a}\varphi(\zeta)\vert^2}\bigg)^{q/p}\ \ud m(\zeta)\Bigg)^{\frac{1}{q}}.\]
 \end{thm}

 The aim of this section is to give the corresponding estimate for the case $p=1.$
 We shall prove that the previous theorem is still valid for $p=1$:

\begin{thm}
Let $u$ be an analytic function on~$\D$ and $\varphi$ an analytic self-map of~$\D.$ Suppose that the weighted composition operator $uC_\varphi$ is bounded from~$H^1$ to~$H^q$ for a certain~$1\leq q<\infty.$ Then we have \[\|uC_\varphi\|_e\approx\limsup_{\vert a\vert\to1^-}\Bigg(\int_{\T}\vert u(\zeta)\vert^q\bigg(\frac{1-\vert a\vert^2}{\vert1-\bar{a}\varphi(\zeta)\vert^2}\bigg)^q\ \ud m(\zeta)\Bigg)^{\frac{1}{q}}.\]
\end{thm}

Let us start with the upper estimate:

\begin{prop}\label{majoration1}
Let $uC_\varphi\in B(H^1,H^q)$ with $1\leq q<\infty.$ Then there exists a positive constant~$\gamma$ such that
\[\|uC_\varphi\|_e\leq \gamma\limsup_{\vert a\vert\rightarrow 1^-}\Bigg(\int_\T\vert u(\zeta)\vert^q\bigg(\frac{1-\vert a\vert^2}{\vert 1-\bar{a}\varphi(\zeta)\vert^2}\bigg)^q\ \ud m(\zeta)\Bigg)^\frac{1}{q}.\]
\end{prop}

The main tool of the proof is the use of Carleson measures.
 Assume that $\mu$ is a finite positive Borel measure on~$\overline{\D}$ and let $1\leq p,q<\infty$. We say that~$\mu$ is a~$(p,q)$-Carleson measure if the embedding~$J_\mu:f\in H^p\mapsto f\in L^q(\mu)$ is well defined. In this case, the closed graph theorem ensures that~$J_\mu$ is continuous. In other words,~$\mu$ is a~$(p,q)$-Carleson measure if there exists a constant~$\gamma_1>0$ such that for every~$f\in H^p,$ 
 \begin{equation}\label{eq6}
 \int_{\overline{\D}}|f(z)|^q\ \ud \mu(z)\leq \gamma_1\|f\|_p^q.
 \end{equation}
 Let $I$ be an arc in $\T$. By~$S(I)$ we denote the Carleson window given by 
 \[S(I)=\{z\in\D\mid\ 1-|I|\leq|z|<1,\ z/|z|\in I\}.\]
 Let us denote by~$\mu_\D$ and~$\mu_\T$ the restrictions of~$\mu$ to~$\D$ and~$\T$ respectively. The following result is a version of a theorem of Duren~(see~\cite{Du}, p.$163$) for measures on~$\overline{\D}$:

 \begin{thm}[see {\cite[Theorem~$2.5$]{BJ}}]\label{p<q Carleson measure}
    Let $1\leq p<q<\infty$. A finite positive Borel measure $\mu$ on $\overline{\D}$ is a~$(p,q)$-Carleson measure if and only if~$\mu_\T=0$ and there exists a constant~$\gamma_2>0$ such that
    \begin{equation}\label{eq7}
    \mu_\D\left(S(I)\right)\leq \gamma_2|I|^{q/p}\quad\textrm{for any arc }I\subset\T.
    \end{equation}
 \end{thm} 
 \noindent Notice that the best constants $\gamma_1$ and $\gamma_2$ in~\eqref{eq6} and~\eqref{eq7} are comparable, meaning that there is a positive constant~$\beta$ independent of the measure~$\mu$ such that~$(1/\beta)\gamma_2\leq \gamma_1\leq \beta \gamma_2$.\\ 
 The notion of Carleson measure was introduced by Carleson in~\cite{Ca} as a part of his work on the corona problem. He gave a characterization of measures~$\mu$ on~$\D$ such that $H^p$ embeds continuously in~$L^p(\mu).$
 
 Examples of such Carleson measures are provided by composition operators. Let~$\varphi:\D\to\D$ be an analytic map and let~$1\leq p,q<\infty.$ The boundedness of the composition operator $C_\varphi:f\mapsto f\circ\varphi$ between $H^p$ and $H^q$ can be rephrased in terms of~$(p,q)$-Carleson measures. Indeed, denote by $m_\varphi$ the \emph{pullback measure} of~$m$ by~$\varphi,$ which is the image  of the Haar measure~$m$ of~$\T$ under the map~$\varphi^*$, defined by \[m_\varphi(A)=m\left(\varphi^{*^{-1}}(A)\right)\] for every Borel subset~$A$ of $\overline{\D}.$ Then \[\|C_\varphi(f)\|_q^q=\int_\T|f\circ\varphi|^q\ \ud m=\int_{\overline{\D}}|f|^q\ \ud m_\varphi=\|J_{m_\varphi
 }(f)\|_q^q\] for all $f\in H^p.$ Thus $C_\varphi$ maps $H^p$ boundedly into $H^q$ if and only if $m_\varphi$ is a $(p,q)$-Carleson measure.\\
 In the sequel we will denote by~$r\D$ the open disk of radius~$r,$ in other words~$r\D=\{z\in\D\mid\ \vert z\vert<r\}$ for $0<r<1.$ We will need the following lemma concerning $(p,q)$-Carleson measures:

\begin{lem}\label{carleson}
Take~$0<r<1$ and let~$\mu$ be a finite positive Borel measure on~$\overline{\D}.$ Let \[N_r^*:=\sup_{\vert a\vert\geq r}\int_{\overline{\D}}\vert k_a(w)\vert^\frac{q}{p}\ \ud\mu(w).\]
If~$\mu$ is a $(p,q)$-Carleson measure for~$1\leq p\leq q<\infty$ then so is $\mu_r:=\mu_{|_{\overline{\D}\backslash r\D}}$. Moreover one can find an absolute constant~$M>0$ satisfying~$\|\mu_r\|\leq MN_r^*$ where $\|\mu_r\|:=\dis\sup_{I\subset\T}\frac{\mu_r\big(S(I)\big)}{\vert I\vert^{q/p}}\cdot$
\end{lem}

We omit the proof of Lemma~\ref{carleson} here, which is a slight modification of the proof of Lemma~$1$ and Lemma~$2$ in~\cite{CZ1} using Theorem~\ref{p<q Carleson measure}.

In the proof of the upper estimate of Theorem~\ref{normeHp-Hq} in~\cite{CZ1}, the authors use a decomposition of the identity on $H^p$ of the form $I=K_N+R_N$ where $K_N$ is the partial sum operator defined by $K_N\left(\sum_{n=0}^\infty a_nz^n\right)=\sum_{n=0}^N a_nz^n,$ and they use the fact that $(K_N)$ is a sequence of compact operators that is uniformly bounded in $B(H^p)$ and that $R_N$ converges pointwise to zero on $H^p.$  Nevertheless the sequence $(K_N)$ is not uniformly bounded in~$B(H^1).$ In fact, $(K_N)$ is uniformly bounded in $B(H^p)$ if and only if the Riesz projection~$P:L^p\rightarrow H^p$ is bounded~\cite[Theorem 2]{Z}, which occurs if and only if~$1<p<\infty.$ Therefore we need to use a different decomposition for the case~$p=1.$ Since $K_N$ is the convolution operator by the Dirichlet kernel on $H^p,$ we shall consider the Fej\'er kernel $F_N$ of order~$N.$ Let us define $K_N:H^1\rightarrow H^1$ to be the convolution operator associated to~$F_N$ that maps~$f\in H^1$ to~$K_Nf=F_N\ast f\in H^1$ and $R_N=I-K_N.$ Then~$\|K_N\|\leq1,\ K_N$ is compact and for every~$f\in H^1,\ \|f-K_Nf\|_1\rightarrow0$ following Fej\'er's theorem. If~$f(z)=\sum_{n\geq0}\hat{f}(n)z^n\in H^1,$ then \[K_Nf(z)=\sum_{n=0}^{N-1}\Big(1-\frac{n}{N}\Big)\hat{f}(n)z^n.\]

\begin{lem}\label{majoration}
Let $1\leq q<\infty$ and suppose that $uC_\varphi\in B(H^1,H^q).$ Then 
\[\|uC_\varphi\|_e\leq\liminf_N\|uC_\varphi R_N\|.\]
\end{lem}

\begin{proof}
\begin{align*}
\|uC_\varphi\|_e & =\|uC_\varphi K_N+uC_\varphi R_N\|_e\\
                 & =\|uC_\varphi R_N\|_e\qquad\qquad\qquad\textrm{since }K_N\textrm{ is compact}\\
                 & \leq \|uC_\varphi R_N\|
\end{align*}
and the result follows taking the lower limit.
\end{proof}

We will need the following lemma for an estimation of the remainder~$R_N$:

\begin{lem}\label{estimate2}
Let $\eps>0$ and $0<r<1.$ Then $\exists N_0=N_0(r)\in\N,\ \forall N\geq N_0,$\[ \vert R_Nf(w)\vert^q<\eps\|f\|_1^q,\] for every $|w|<r$ and for every $f$ in $H^1.$
\end{lem}

\begin{proof}
Let $K_w(z)=1/(1-\bar{w}z), \ w\in\D,\ z\in\D.\ K_w$ is a bounded analytic function on~$\D.$ It is easy to see that for every~$f\in H^1,$\[\left\langle R_Nf,K_w\right\rangle=\left\langle f,R_NK_w\right\rangle\] where~$\vert w\vert<r,\ N\geq1$ and  \[\left\langle f,g\right\rangle=\frac{1}{2\pi}\int_0^{2\pi}f(e^{i\theta})\overline{g(e^{i\theta})}\ \ud\theta\]for $f\in H^1$ and $g\in H^\infty.$ 
Then we have~$\vert R_Nf(w)\vert=\vert\left\langle R_Nf,K_w\right\rangle\vert=\vert\left\langle f,R_NK_w\right\rangle\vert\leq\|f\|_1\|R_NK_w\|_\infty.$ Take $\vert w\vert< r$ and choose $N_0\in\N$ so that for every $N\geq N_0$ one has $r^N\leq\eps^{1/q}(1-r)/2$ and~$1/N\sum_{n=1}^{N-1}nr^n\leq(1/2)\eps^{1/q}.$ Since 
\[R_NK_w(z)=R_N\bigg(\sum_{n=0}^\infty\bar{w}^nz^n\bigg)=\sum_{n=0}^{N-1}\frac{n}{N}\bar{w}^nz^n+\sum_{n=N}^\infty\bar{w}^nz^n,\]
one has 
\[\| R_NK_w\|_\infty<\frac{1}{N}\sum_{n=0}^{N-1}nr^n+\sum_{n=N}^\infty r^n\leq\eps^{1/q}.\]

Thus~$\vert R_Nf(w)\vert^q\leq\eps\|f\|_1^q$ for every $f$ in $H^1.$
\end{proof}

{\it Proof of Proposition~\ref{majoration1}. }
Denote by $\mu$ the measure which is absolutely continuous with respect to~$m$ and whose density is~$\vert u\vert^q,$ and let~$\mu_\varphi= \mu\circ\varphi^{-1}$ be the pullback measure of~$\mu$ by~$\varphi.$
Fix~$0<r<1$. For every~$f\in H^1,$ we have
\begin{align}
\|(uC_\varphi R_N)f\|_q^q &=\int_\T\vert              u(\zeta)\vert^q\big\vert\big((R_Nf)\circ\varphi\big)(\zeta)\big\vert^q\ \ud m(\zeta) \nonumber \\                             &=\int_\T \big\vert\big((R_Nf)\circ\varphi\big)(\zeta)\big\vert^q\ \ud\mu(\zeta) \nonumber\\
            &=\int_{\overline{\D}}\vert R_Nf(w)\vert^q\ \ud\mu_\varphi(w) \nonumber\\
            &=\int_{\overline{\D}\backslash r\D}\vert
            R_Nf(w)\vert^q\ \ud\mu_\varphi(w)+\int_{r\D}\vert R_Nf(w)\vert^q\ \ud\mu_\varphi(w) \nonumber\\
            &=I_1(N,r,f)+I_2(N,r,f). \label{eq2}
\end{align}

Let us first show that $\dis\lim_N\dis\sup_{\|f\|_1=1}I_2(N,r,f)=0$.
For~$\eps>0,$ Lemma~\ref{estimate2} gives us an integer~$N_0(r)$ such that for every~$N\geq N_0(r),$
\begin{align*}
I_2(N,r,f) &=\int_{r\D}\vert R_Nf(w)\vert^q\ \ud\mu_\varphi(w)\\
           &\leq\eps\|f\|_1^q\mu_\varphi(r\D)\\
           &\leq\eps\|f\|_1^q\mu_\varphi(\overline{\D})\\
           &\leq\eps\|f\|_1^q\|u\|_q^q.
\end{align*}
So, $r$ being fixed, we have~$\dis\lim_N\dis\sup_{\|f\|_1=1}I_2(N,r,f)=0.$\par
Now we need an estimate of~$I_1(N,r,f).$
The continuity of~$uC_\varphi:H^1\rightarrow H^q$ ensures that~$\mu_\varphi$ is a~$(1,q)$-Carleson measure, and therefore~$\mu_{\varphi,r}:=\mu_{\varphi_{|_{\overline{\D}\backslash r\D}}}$ is also a~$(1,q)$-Carleson measure by using Lemma~\ref{carleson} for $p=1.$ It follows that  
\begin{align*}
\int_{\overline{\D}\backslash r\D}\vert R_Nf(w)\vert^q\ \ud\mu_{\varphi,r}(w)        &\leq\gamma_1\|R_Nf\|_1^q\\
&\leq \beta\|\mu_{\varphi,r}\|\|R_Nf\|_1^q\\
&\leq 2^q\beta MN_r^*\|f\|_1^q
\end{align*}
using Lemma~\ref{carleson} and the fact that~$\|R_N\|\leq1+\|K_N\|\leq2$ for every~$N\in\N.$ We take the supremum over~$B_{H^1}$ and take the lower limit as~$N$ tends to infinity in~\eqref{eq2} to obtain\[\liminf_{N\to\infty}\|uC_\varphi R_N\|^q\leq2^q\beta MN_r^*.\]
Now as~$r$ goes to~$1$ we have:
\begin{align*}
\lim_{r\to1}N_r^* &=\limsup_{\vert a\vert\to1^-}\int_{\overline{\D}}\vert k_a(w)\vert^q\ \ud\mu_\varphi(w)\\
                  &=\limsup_{\vert a\vert\to1^-}\int_{\T}\vert u(\zeta)\vert^q\bigg(\frac{1-\vert a\vert^2}{\vert1-\bar{a}\varphi(\zeta)\vert^2}\bigg)^q\ \ud m(\zeta)
\end{align*}
and we obtain the estimate announced using~Lemma~\ref{majoration}. \CQFD

Now let us turn to the lower estimate in Theorem~\ref{normeHp-Hq}. Let $1\leq q<\infty.$ Consider~$F_N$ the Fej\'er kernel of order~$N,$ and define~$K_N:H^q\rightarrow H^q$ the convolution operator associated to~$F_N$ and $R_N=I-K_N$. Then~$(K_N)_N$ is a sequence of uniformly bounded compact operators in $B(H^q)$, and $\|R_Nf\|_q\rightarrow0$ for all~$f\in H^q.$

\begin{lem}\label{minoration}
There exists~$0<\gamma\leq 2$ such that whenever $uC_\varphi$ is a bounded operator from~$H^1$ to~$H^q$ with~$1\leq q<\infty,$ one has\[\frac{1}{\gamma}\limsup_N\|R_NuC_\varphi\|\leq\|uC_\varphi\|_e.\]
\end{lem}

\begin{proof} Take $K\in B(H^1,H^q)$ a compact operator. Since $(K_N)$ is uniformly bounded, one can find~$\gamma>0$ satisfying~$\|R_N\|\leq 1+\|K_N\|\leq \gamma$ for all $N>0,$ and we have: 
\begin{align*}
\|uC_\varphi+K\| & \geq\frac{1}{\gamma}\|R_N(uC_\varphi+K)\| \\
                 & \geq\frac{1}{\gamma}\|R_NuC_\varphi\|-\frac{1}{\gamma}\|R_NK\|.
\end{align*}
Now use the fact that~$(R_N)$ goes pointwise to zero in~$H^q$, and consequently~$(R_N)$ converges strongly to zero over the compact set~$\overline{K(B_{H^1})}$ as~$N$ goes to infinity. It follows that~$\|R_NK\|\dis\mathop{\longrightarrow}_N0,$ and 
\[\|uC_\varphi+K\|\geq\frac{1}{\gamma}\limsup_N\|R_NuC_\varphi\|\]
for every compact operator~$K:H^1\rightarrow H^q.$
\end{proof}

\begin{prop}\label{minoration1}
Let $u$ be an analytic function on~$\D$ and $\varphi$ an analytic self-map of~$\D.$ Assume that $uC_\varphi\in B(H^1,H^q)$ with $1\leq q<\infty.$ Then
\[\|uC_\varphi\|_e\geq\frac{1}{\gamma}\limsup_{\vert a\vert\rightarrow 1^-}\Bigg(\int_\T\vert u(\zeta)\vert^q\bigg(\frac{1-\vert a\vert^2}{\vert 1-\bar{a}\varphi(\zeta)\vert^2}\bigg)^q\ \ud m(\zeta)\Bigg)^\frac{1}{q}.\]
\end{prop}

\begin{proof}
Since $k_a$ is a unit vector in~$H^1,$
\begin{equation}\label{eq1}
\|R_NuC_\varphi\|=\|uC_\varphi-K_NuC_\varphi\|\geq\|uC_\varphi k_a\|_q-\|K_NuC_\varphi k_a\|_q.
\end{equation}

\textit{First case: }$q>1$\\
Since $(k_a)$ converges to zero for the topology of uniform convergence on compact sets in~$\D$ as~$\vert a\vert$ goes to~$1,$ so does~$uC_\varphi(k_a)$. The topology of uniform convergence on compact sets in~$\D$ and the weak topology agree on $H^q$, therefore it follows that~$uC_\varphi(k_a)$ goes to zero for the weak topology in~$H^q$ as~$\vert a\vert$ goes to~$1.$ Since $K_N$ is a compact operator, it is completely continuous and carries weak-null sequences to norm-null sequences. So~$\|K_N\big(uC_\varphi(k_a)\big)\|_q\rightarrow0$ when $\vert a\vert\rightarrow1$, and
\[\|R_NuC_\varphi\|\geq\limsup_{\vert a\vert\rightarrow1^-}\|uC_\varphi(k_a)\|_q.\]
Taking the upper limit as~$N\rightarrow\infty$, we obtain the result using~Lemma~$\ref{minoration}.$\\
For the second case we will need the following computational lemma:

\begin{lem}\label{estimate-coeff} 
Let~$\varphi$ be an analytic self-map of~$\D.$ Take $a\in\D$ and $N\geq1$ an integer. Denote by~$\alpha_p(a)$ the~$p$-th Fourier coefficient of~$C_\varphi\left(k_a/(1-|a|^2)\right)$, so that for every~$z\in\D$ we have \[k_a\big(\varphi(z)\big)=(1-|a|^2)\sum_{p=0}^\infty\alpha_p(a)z^p.\]
Then there exists a positive constant $M=M(N)>0$ depending on~$N$ such that~\mbox{$|\alpha_p(a)|\leq M$} for every~$p\leq N$ and every~$a\in\D.$ 
\end{lem}

\begin{proof} 
Write $\varphi(z)=a_0+\psi(z)$ with $a_0=\varphi(0)\in\D$ and $\psi(0)=0.$ If we develop $k_a(z)$ as a Taylor series and replace $z$ by $\varphi(z)$ we obtain:\[k_a\big(\varphi(z)\big)=(1-|a|^2)\sum_{n=0}^\infty(n+1)(\bar{a})^n\varphi(z)^n.\]  Then 
\begin{align*}
\alpha_p(a) &=\left\langle \sum_{n=0}^\infty(n+1)(\bar{a})^n\varphi(z)^n,z^p\right\rangle\\
         &=\sum_{n=0}^\infty(n+1)(\bar{a})^n\sum_{j=0}^n\binom{n}{j}a_0^{n-j}\left\langle \psi(z)^j,z^p\right\rangle.\\
\end{align*}
where $\langle f,g\rangle=\int_\T f\bar{g}\ \ud m.$ Note that $\left\langle \psi(z)^j,z^p\right\rangle=0$ if $j>p$ since~$\psi(0)=0,$ and consequently 
\begin{align*}
\alpha_p(a) &=\sum_{n=0}^\infty(n+1)(\bar{a})^n\sum_{j=0}^{\min(n,p)}\binom{n}{j}a_0^{n-j}\left\langle \psi(z)^j,z^p\right\rangle\\
         &=\sum_{j=0}^p\sum_{n=j}^\infty(n+1)(\bar{a})^n\binom{n}{j}a_0^{n-j}\left\langle \psi(z)^j,z^p\right\rangle\\
         &=\sum_{j=0}^p\left\langle \psi(z)^j,z^p\right\rangle\sum_{n=j}^\infty(n+1)(\bar{a})^n\binom{n}{j}a_0^{n-j}.
\end{align*}
In the case where~$a_0\neq0$ we obtain 
\begin{align*}
\alpha_p(a) &=\sum_{j=0}^p\left\langle \psi(z)^j,z^p\right\rangle a_0^{-j}\sum_{n=j}^\infty(n+1)\binom{n}{j}(\bar{a}a_0)^n\\
         &=\sum_{j=0}^p\left\langle \psi(z)^j,z^p\right\rangle a_0^{-j}\frac{(j+1)(\bar{a}a_0)^j}{(1-\bar{a}a_0)^{j+2}}\\
         &=\sum_{j=0}^p\left\langle \psi(z)^j,z^p\right\rangle \frac{(j+1)(\bar{a})^j}{(1-\bar{a}a_0)^{j+2}}
\end{align*}
using the following equalities for $x=\bar{a}a_0\in\D$: \[\sum_{n=j}^\infty(n+1)\binom{n}{j}x^n=\left(\sum_{n=j}^\infty\binom{n}{j}x^{n+1}\right)'=\left(\frac{x^{j+1}}{(1-x)^{j+1}}\right)'=\frac{(j+1)x^j}{(1-x)^{j+2}}\]
Note that the last expression obtained for~$\alpha_p(a)$ is also valid for~$a_0=0.$ Thus, for~$0\leq p\leq N$ we have the following estimates: 
\begin{align*}
|\alpha_p(a)| &\leq\sum_{j=0}^p|\left\langle \psi(z)^j,z^p\right\rangle|\frac{j+1}{(1-|a_0|)^{j+2}}\\
           &\leq\sum_{j=0}^p\|\psi^j\|_\infty\frac{N+1}{(1-|a_0|)^{N+2}}\\
           &\leq\frac{(N+1)^2}{(1-|a_0|)^{N+2}}\max_{0\leq j\leq N}\|\psi^j\|_\infty\\
           &\leq M,      
\end{align*}
where $M$ is a constant independent from~$a.$
\end{proof}

\textit{Second case: }$q=1$\\
In this case, it is no longer for the weak topology but for the weak-star topology of~$H^1$ that $uC_\varphi(k_a)$ tends to zero when~$|a|\to1.$ Nevertheless, it is still true that~$\|K_NuC_\varphi(k_a)\|_1\rightarrow0$ as~$\vert a\vert\rightarrow1.$ Indeed if~$f(z)=\sum_{n\geq0}\hat{f}(n)z^n\in H^1,$ then \[K_Nf(z)=\sum_{n=0}^{N-1}\left(1-\frac{n}{N}\right)\hat{f}(n)z^n.\]
We have the following development:\[k_a\big(\varphi(z)\big)=(1-|a|^2)\sum_{n=0}^\infty\alpha_n(a)z^n.\]
Denote by~$u_n$ the $n$-th Fourier coefficient of $u$, so that \[uC_\varphi(k_a)(z)=(1-\vert a\vert^2)\sum_{n=0}^\infty\bigg(\sum_{p=0}^n\alpha_p(a)u_{n-p}\bigg)z^n,\ \forall z\in\D.\]
It follows that 
\[\|K_NuC_\varphi(k_a)\|_1 \leq(1-\vert a\vert^2)\sum_{n=0}^{N-1}\left(1-\frac{n}{N}\right)\bigg\vert\sum_{p=0}^n\alpha_p(a)u_{n-p}\bigg\vert\|z^n\|_1.\]
Now using estimates from Lemma~\ref{estimate-coeff}, one can find a constant~$M>0$ independent from $a$ such that~$\vert\alpha_p(a)\vert\leq M$ for every~$a\in\D$ and~$0\leq p\leq N-1.$ Use the fact that~$\|z^n\|_1=1$ and~$\vert u_p\vert\leq\|u\|_1$ to deduce that there is a constant~$M'>0$ independent from $a$ such that \[\|K_NuC_\varphi(k_a)\|_1\leq M'(1-\vert a\vert^2)\|u\|_1\] for all~$a\in\D.$ Thus~$K_NuC_\varphi(k_a)$ converges to zero in~$H^1$ when~$\vert a\vert\rightarrow1,$ and take the upper limit in~\ref{eq1} when~$a$ tends to~$1^-$ to obtain\[\|R_NuC_\varphi\|\geq\limsup_{\vert a\vert\rightarrow1}\|uC_\varphi(k_a)\|_1,\quad\forall N\geq0.\]
We conclude with Lemma~\ref{minoration} and observe that~$\gamma=\sup\|R_N\|\leq2$ since $\|R_N\|\leq1+\|K_N\|\leq 2.$
\end{proof}


\section{$uC_\varphi\in B(H^p,H^\infty)$ for $1\leq p<\infty$}
Let $u$ be a bounded analytic function. Characterizations of boundedness and compactness of $uC_\varphi$ as a linear map between $H^p$ and $H^\infty$  have been studied in~\cite{CH} for~$p\geq1.$ Indeed, \[uC_\varphi\in B(H^p,H^\infty)\textrm{ if and only if } \sup_{z\in\D}\frac{|u(z)|^p}{1-|\varphi(z)|^2}<\infty\] and \[uC_\varphi\textrm{ is compact if and only if }\|\varphi\|_\infty<1\textrm{ or } \lim_{|\varphi(z)|\rightarrow1}\frac{|u(z)|^p}{1-|\varphi(z)|^2}=0.\]
In the case where $\|\varphi\|_\infty=1$ we let \[M_\varphi(u)=\limsup_{|\varphi(z)|\rightarrow1} \frac{|u(z)|}{(1-|\varphi(z)|^2)^{\frac{1}{p}}}\cdot\] 
 As regarding Theorem~$1.7$ in~\cite{Le}, it seems reasonable to think that the essential norm of $uC_\varphi$ is equivalent to the quantity $M_\varphi(u).$ We first have a majorization:

\begin{prop}\label{majoration2}
Let $u$ be an analytic function on~$\D$ and $\varphi$ an analytic self-map of~$\D.$ Suppose that $uC_\varphi$ is a bounded operator from~$H^p$ to~$H^\infty,$ where~$1\leq p<\infty$ and that~$\|\varphi\|_\infty=1.$ Then\[\|uC_\varphi\|_e\leq2M_\varphi(u).\]
\end{prop}

\begin{proof}Let $\eps$ be a real positive number, and pick~$r<1$ satisfying\[\sup_{|\varphi(z)|\geq r}\frac{|u(z)|}{(1-|\varphi(z)|^2)^{\frac{1}{p}}}\leq M_\varphi(u)+\eps.\]We approximate $uC_\varphi$ by $uC_\varphi K_N$ where $K_N:H^p\to H^p$ is the convolution operator by the Fej\'er kernel of order~$N,$ where $N$ is chosen so that $|R_Nf(w)|<\eps\|f\|_1$ for every $f\in H^1$ and every $|w|<r$ (Lemma~\ref{estimate2} for \mbox{$q=1$}). We want to show that~$\|uC_\varphi-uC_\varphi K_N\|=\|uC_\varphi R_N\|\leq\max(2M_\varphi(u)+2\eps,\eps\|u\|_\infty),$ which will prove our assertion. If~$f$ is a unit vector in~$H^p$, then the norm of $uC_\varphi R_N(f)$ is equal to\[\max\left(\sup_{|\varphi(z)|\geq r}|u(z)(R_Nf)\circ\varphi(z)|,\sup_{|\varphi(z)|< r}|u(z)(R_Nf)\circ\varphi(z)|\right).\]
We want to estimate the first term. If $\omega\in\D,$ we denote by $\delta_\omega$ the linear functional on~$H^p$ defined by $\delta_\omega(f)=f(\omega).$ Then~$\delta_\omega\in(H^p)^*$ and $\|\delta_{w}\|_{(H^p)^*}=1/(1-|w|^2)^{1/p}$ for every~$w\in\D.$ Therefore  
\begin{align*}
\sup_{|\varphi(z)|\geq r}|u(z)(R_Nf)\circ\varphi(z)| &\leq\sup_{|\varphi(z)|\geq r}|u(z)|\|\delta_{\varphi(z)}\|_{(H^p)^*}\|R_Nf\|_p\\
 &\leq2\sup_{|\varphi(z)|\geq r}\frac{|u(z)|}{(1-|\varphi(z)|^2)^{\frac{1}{p}}}\\
 &\leq2\left(M_\varphi(u)+\eps\right),
\end{align*}
using the fact that $\|R_N f\|_p\leq2.$\\
For the second term, since~$|\varphi(z)|< r$ we have 
\[\left|u(z)R_N f\left(\varphi(z)\right)\right|\leq\|u\|_\infty|R_N f\left(\varphi(z)\right)|\leq\eps\|u\|_\infty\|f\|_1\leq\eps\|u\|_\infty
\]
which ends the proof.
\end{proof}

On the other hand, we have the lower estimate:

\begin{prop}\label{minoration2}
Let $u$ be an analytic function on~$\D$ and $\varphi$ an analytic self-map of~$\D$ satisfying~$\|\varphi\|_\infty=1.$ Suppose that $uC_\varphi$ is a bounded operator from~$H^p$ to~$H^\infty,$ where $1\leq p<\infty.$ Then \[\frac{1}{2}M_\varphi(u)\leq\|uC_\varphi\|_e.\]
\end{prop}

\begin{proof}
Assume that $uC_\varphi$ is not compact, implying~$M_\varphi(u)>0.$ Let~$(z_n)$ be a sequence in~$\D$ satisfying \[\lim_n|\varphi(z_n)|=1\quad\textrm{ and }\quad \lim_n\frac{|u(z_n)|}{(1-|\varphi(z_n)|^2)^{\frac{1}{p}}}=M_\varphi(u).\]

Consider the sequence $(f_n)$ defined by \[f_n(z)=k_{\varphi(z_n)}(z)^{1/p}=\frac{\left(1-|\varphi(z_n)|^2\right)^{\frac{1}{p}}}{\left(1-\overline{\varphi(z_n)}z\right)^{\frac{2}{p}}}\cdot\]
Each $f_n$ is a unit vector of~$H^p.$ Let~$K:H^p\to H^\infty$ be a compact operator.\\

\textit{First case: $p>1$}\\
Since the sequence $(f_n)$ converges to zero for the weak topology of~$H^p$ and~$K$ is completely continuous, the sequence~$(Kf_n)$ converges to zero for the norm topology in~$H^\infty.$ Use that \mbox{$\|uC_\varphi+K\|\geq\|uC_\varphi(f_n)\|_\infty-\|Kf_n\|_\infty$} and take the upper limit when~$n$ tends to infinity to obtain
 \begin{align*}
\|uC_\varphi+K\| &\geq\limsup_n\|uC_\varphi(f_n)\|_\infty\\
                 &\geq\limsup_n|u(z_n)|\left|f_n\left(\varphi(z_n)\right)\right|\\
                 &\geq\limsup_n\frac{|u(z_n)|}{\left(1-|\varphi(z_n)|^2\right)^{\frac{1}{p}}}\\
                 &\geq M_\varphi(u).
\end{align*}

\textit{Second case: $p=1$}\\
Let~$\eps>0.$ Since the sequence $(f_n)$ is no longer weakly convergent to zero in~$H^1,$ we cannot assert that~$(Kf_n)_n$ goes to zero in~$H^\infty.$ Nevertheless, passing to subsequences, one can assume that~$(Kf_{n_k})_k$ converges in~$H^\infty,$ and hence is a Cauchy sequence. So we can find an integer~$N>0$ such that for every~$k$ and~$m$ greater than~$N$ we have~$\|Kf_{n_k}-Kf_{n_m}\|<\eps.$ We deduce that  
\begin{align*}
\|uC_\varphi+K\| &\geq\left\|(uC_\varphi+K)\left(\frac{f_{n_k}-f_{n_m}}{2}\right)\right\|_\infty\\
                 &\geq\frac{1}{2}\|uC_\varphi(f_{n_k}-f_{n_m})\|_\infty-\frac{\eps}{2}\\
                 &\geq\frac{1}{2}|u(z_{n_k})|\left|f_{n_k}\left(\varphi(z_              {n_k})\right)-f_{n_m}\left(\varphi(z_{n_k})\right)\right|-\frac{\eps}{2}\\
                 &\geq\frac{|u(z_{n_k})|}{2\left(1-|\varphi(z_{n_k})|^2\right)}-\frac{|u(z_{n_k})|\left(1-|\varphi(z_{n_m})|^2\right)}{2\left|1-\overline{\varphi(z_{n_m})}\varphi(z_{n_k})\right|^2}-\frac{\eps}{2}
\end{align*}
Now take the upper limit as~$m$ goes to infinity ($k$ being fixed) and recall that~$\lim_m|\varphi(z_{n_m})|=1$ and~$|\varphi(z_{n_k})|<1$ to obtain                 
\[\|uC_\varphi+K\|\geq\frac{|u(z_{n_k})|}{2\left(1-|\varphi(z_{n_k})|^2\right)}-\frac{\eps}{2}\]
for every $k\geq N.$ It remains to make~$k$ tend to infinity to have \[\|uC_\varphi+K\|\geq\frac{1}{2}M_\varphi(u)-\frac{\eps}{2}.\]
\end{proof}

Combining Proposition~\ref{majoration2} and Proposition~\ref{minoration2} we obtain the following estimate:

\begin{thm}
Let $u$ be an analytic function on~$\D$ and $\varphi$ an analytic self-map of~$\D$ satisfying~$\|\varphi\|_\infty=1.$ Suppose that $uC_\varphi$ is a bounded operator from~$H^p$ to~$H^\infty,$ where~$1\leq p<\infty.$ Then~$\|uC_\varphi\|_e\approx M_\varphi(u).$ More precisely, we have the following inequalities: \[\frac{1}{2}M_\varphi(u)\leq\|uC_\varphi\|_e\leq2M_\varphi(u).\]
\end{thm}
Note that if $p>1$ one can replace the constant~$1/2$ by~$1.$


\section{$uC_\varphi\in B(H^\infty,H^q)$ for $\infty>q\geq 1$}

In this setting, boundedness of the weighted composition operator~$uC_\varphi$ is equivalent to saying that~$u$ belongs to~$H^q$, and~$uC_\varphi$ is compact if and only if $u=0$ or~$\vert E_\varphi\vert=0$ where~$E_\varphi=\{\zeta\in\T\mid\ \varphi^*(\zeta)\in\T\}$ is the extremal set of~$\varphi$ (see~\cite{CH}). We give here some estimates of the essential norm of~$uC_\varphi$ that appear in~\cite{GM} for the special case of composition operators:

\begin{thm}
Let $u\in H^q,$ with $\infty>q\geq1$ and $\varphi$ be an analytic self-map of~$\D.$ Then $\|uC_\varphi\|_e\approx\left(\int_{E_\varphi}\vert u(\zeta)\vert^q\ \ud m(\zeta)\right)^\frac{1}{q}.$ More precisely, 
\[\frac{1}{2}\left(\int_{E_\varphi}\vert u(\zeta)\vert^q\ \ud m(\zeta)\right)^\frac{1}{q}\leq\|uC_\varphi\|_e\leq2\left(\int_{E_\varphi}\vert u(\zeta)\vert^q\ \ud m(\zeta)\right)^\frac{1}{q}.\]
\end{thm}

We start with the upper estimate:

\begin{prop}
Let $u\in H^q,$ with $\infty>q\geq1$ and $\varphi$ be an analytic self-map of~$\D.$ Then \[\|uC_\varphi\|_e\leq2\bigg(\int_{E_\varphi}\vert u(\zeta)\vert^q\ \ud m(\zeta)\bigg)^\frac{1}{q}.\]
\end{prop}

\begin{proof}Take $0<r<1.$ Since $\|r\varphi\|_\infty\leq r<1$, the set $E_{r\varphi}$ is empty and therefore the operator~$uC_{r\varphi}$ is compact. Thus~$\|uC_\varphi\|_e\leq\|uC_\varphi-uC_{r\varphi}\|.$ But
\begin{equation}\label{eq5}
\|uC_\varphi-uC_{r\varphi}\|^q=\sup_{\|f\|_\infty\leq1}\int_\T\vert u(\zeta)\vert^q\big\vert f\big(\varphi(\zeta)\big)-f\big(r\varphi(\zeta)\big)\big\vert^q\ \ud m(\zeta).
\end{equation}
If $|E_\varphi|=1$ then the integral in~\eqref{eq5} coincides with \[\int_{E_\varphi}\vert u(\zeta)\vert^q\big\vert f\big(\varphi(\zeta)\big)-f\big(r\varphi(\zeta)\big)\big\vert^q\ \ud m(\zeta)\] which is less than $2^q\int_{E_\varphi}\vert u(\zeta)\vert^q\ \ud m(\zeta).$  If $|E_\varphi|<1$ we let $F_\eps=\{\zeta\in\T\mid\ \vert\varphi^*(\zeta)\vert<1-\eps\}$ for $\eps>0,$ which is a nonempty set for $\eps$ sufficiently small. (Let us mention here that an element $\zeta\in\T$ needs not to satisfy neither $\zeta\in E_\varphi$ nor $\zeta\in\bigcup_{\eps>0} F_\eps.$ It can happen that the radial limit~$\varphi^*(\zeta)$ does not exist, but this happens only for~$\zeta$ belonging to a set of measure zero). We will use the pseudohyperbolic distance~$\rho$ defined for~$z$ and~$w$ in the unit disk by~$\rho(z,w)=\vert z-w\vert/\vert 1-\bar{w}z\vert.$ The Pick-Schwarz's theorem ensures that~$\rho\big(f(z),f(w)\big)\leq\rho(z,w)$ for every function~$f\in B_{H^\infty}.$ As a consequence the inequality $\vert f(z)-f(w)\vert\leq2\rho(z,w)$ holds for every $w$ and~$z$ in~$\D.$\\
If $\zeta$ is an element of~$F_\eps$ then\[\rho\big(\varphi(\zeta),r\varphi(\zeta)\big)=\frac{(1-r)\vert\varphi(\zeta)\vert}{1-r\vert\varphi(\zeta)\vert^2}\leq\frac{1-r}{1-r(1-\eps)^2}.\]
One can choose $0<r<1$ satisfying~$\sup_{F_\eps}\rho\big(\varphi(\zeta),r\varphi(\zeta)\big)<\eps/2,$ and therefore 
\[\big\vert f\big(\varphi(\zeta)\big)-f\big(r\varphi(\zeta)\big)\big\vert\leq2\sup_{F_\eps}\rho\big(\varphi(\zeta),r\varphi(\zeta)\big)\leq\eps\]for all $\zeta\in F_\eps$ and for every function~$f$ in the closed unit ball of~$H^\infty.$ It follows from these estimates and~\eqref{eq5} that
\begin{align*}
\|uC_\varphi-uC_{r\varphi}\|^q&\leq\sup_{\|f\|_\infty\leq1}\bigg(\int_{F_\eps}\vert u(\zeta)\vert^q\eps^q\ \ud m(\zeta)+\int_{\T\backslash F_\eps}2^q\vert u(\zeta)\vert^q\ \ud m(\zeta)\bigg)\\
                              &\leq\eps^q\|u\|_q^q+2^q\int_{\T\backslash F_\eps}\vert u(\zeta)\vert^q\ \ud m(\zeta).
\end{align*}
Make $\eps$ tend to zero to deduce the upper estimate.
\end{proof}

Let us turn to the lower estimate:

\begin{prop}
Suppose that $\varphi$ is an analytic self-map of~$\D$ and $u\in H^q$ with $\infty>q\geq1.$ Then \[\|uC_\varphi\|_e\geq\frac{1}{2}
\left(\int_{E_\varphi}\vert u(\zeta)\vert^q\ \ud m(\zeta)\right)^\frac{1}{q}.\]
\end{prop}

\begin{proof}
Take a compact operator $K\in B(H^\infty,H^q).$ Since the sequence~$(z^n)_{n\in\N}$ is bounded in~$H^\infty$, there exists an increasing sequence of integers~$(n_k)_{k\geq0}$ such that~$\left(K(z^{n_k})\right)_{k\geq0}$ converges in~$H^q.$ For any $\eps>0$ one can find~$N\in\N$ such that for every~$k,m\geq N$ we have~\mbox{$\|Kz^{n_k}-Kz^{n_m}\|_q<\eps.$} If $0<r<1,$ we let $g_r(z)=g(rz)$ for a function~$g$ defined on~$\D.$ Take~$k\geq N.$ Then there exists~$0<r<1$ such that \[\|\left(u\varphi^{n_k}\right)_r\|_q\geq\|u\varphi^{n_k}\|_q-\eps.\]
For all $m\geq N$ we have 
\begin{align*}
\|uC_\varphi+K\| &\geq\left\|(uC_\varphi+K)\left(\frac{z^{n_k}-z^{n_m}}{2}\right)\right\|_q\\
&\geq\frac{1}{2}\left\|u\left(\varphi^{n_k}-\varphi^{n_m}\right)\right\|_q-\frac{\eps}{2}\\           &\geq\frac{1}{2}\left\|\left(u\varphi^{n_k}\right)_r-\left(u\varphi^{n_m}\right)_r\right\|_q-\frac{\eps}{2}\\
&\geq\frac{1}{2}\left(\left\|\left(u\varphi^{n_k}\right)_r\right\|_q-\left\|\left(u\varphi^{n_m}\right)_r\right\|_q\right)-\frac{\eps}{2}\\
&\geq\frac{1}{2}\left(\left\|u\varphi^{n_k}\right\|_q-\left\|\left(u\varphi^{n_m}\right)_r\right\|_q\right)-\eps.\\
\end{align*}
Let us make $m$ tend to infinity, keeping in mind that $0<r<1$ and $\|\varphi_r\|_\infty<1$: \[\|\left(u\varphi^{n_m}\right)_r\|_q\leq \|u\|_q\|(\varphi_r)^{n_m}\|_\infty\leq \|u\|_q\|\varphi_r\|_\infty^{n_m}\mathop{\longrightarrow}_m0.\] Thus $\|uC_\varphi+K\|\geq(1/2)\|u\varphi^{n_k}\|_q-\eps$ for all~$k\geq N.$ We conclude noticing that
\[\|u\varphi^{n_k}\|_q=\left(\int_\T\left|u(\zeta)\varphi(\zeta)^{n_k}\right|^q\ \ud m(\zeta)\right)^{\frac{1}{q}}\mathop{\longrightarrow}_k\left(\int_{E_\varphi}|u(\zeta)|^q\ \ud m(\zeta)\right)^{\frac{1}{q}}.\]
\end{proof}


\section{$uC_\varphi\in B(H^p,H^q)$ for $\infty>p>q\geq1$}

In~\cite{GM}, the authors give an estimate of the essential norm of a composition operator between~$H^p$ and~$H^q$ for~$1<q<p<\infty.$ The proof makes use of the Riesz projection from~$L^q$ onto~$H^q$, which is a bounded operator for~$1<q<\infty.$ Since it is not bounded from~$L^1$ to~$H^1$ ($H^1$ is not even complemented in~$L^1$) there is no way to use a similar argument. So we need a different approach to get some estimates for~$q=1$. A solution is to make use of Carleson measures. First, we give a characterization of the boundedness of~$uC_\varphi$ in terms of a Carleson measure. In the case where~$p>q$, Carleson measures on~$\overline{\D}$ are characterized in~\cite{BJ}. Denote by~$\Gamma(\zeta)$ the Stolz domain generated by~$\zeta\in\T,$ \emph{i.e.} the interior of the convex hull of the set~$\{\zeta\}\cup(\alpha\D)$, where~$0<\alpha<1$ is arbitrary but fixed.

   \begin{thm}[see {\cite[Theorem~$2.2$]{BJ}}]\label{p>q Carleson measure}
   Let $\mu$ be a measure on $\overline{\D}$, $1\leq q<p<\infty$ and $s=p/(p-q).$ Then~$\mu$ is a~$(p,q)$-Carleson measure on~$\overline{\D}$ if and only if~$\zeta\mapsto\int_{\Gamma(\zeta)}\dis\frac{\ud\mu(z)}{1-|z|^2}$ belongs to~$L^s(\T)$ and~$\mu_\T=F\ud m$ for a function~$F\in L^s(\T).$
   \end{thm}  
   
   This leads to a characterization of the continuity of a weighted composition operator between~$H^p$ and~$H^q$:

\begin{cor}\label{contcarleson}
Let $u$ be an analytic function on~$\D$ and $\varphi$ an analytic self-map of~$\D.$ For $1\leq q<p<\infty,$ the weighted composition operator $uC_\varphi:H^p\rightarrow H^q$ is bounded if and only if $G:\zeta\in\T\mapsto G(\zeta)=\int_{\Gamma(\zeta)}\frac{\ud \mu_\varphi(z)}{1-|z|^2}$ belongs to~$L^s(\T)$ for~$s=p/(p-q)$ and~$\mu_{\varphi_{|_\T}}=F\ud m$ for a certain~$F\in L^s(\T)$, where~$\ud\mu=|u|^q\ud m$ and~$\mu_\varphi=\mu\circ\varphi^{-1}$ is the pullback measure of~$\mu$ by~$\varphi.$
\end{cor}

\begin{proof}$uC_\varphi$ is a bounded operator if and only if there exists~$\gamma>0$ such that for any~$f\in H^p,\ \int_\T|u(\zeta)|^q\left|f\circ\varphi(\zeta)\right|^q\ \ud m(\zeta)\leq \gamma\|f\|_p^q,$ which is equivalent (via a change of variables) to~$\int_{\overline{\D}}\left|f(z)\right|^q\ \ud \mu_\varphi(z)\leq \gamma\|f\|_p^q$ for every~$f\in H^p.$ This exactly means that~$\mu_\varphi$ is a~$(p,q)$-Carleson measure. This is equivalent by Theorem~\ref{p>q Carleson measure} to the condition announced.
\end{proof}

If $f\in H^p,$ the Hardy-Littlewood maximal nontangential function $Mf$ is defined by $Mf(\zeta)=\sup_{z\in\Gamma(\zeta)}|f(z)|$ for $\zeta\in\T.$ For $1<p<\infty,$ $M$ is a bounded operator from $H^p$ to $L^p$ and we will denote its norm by~$\|M\|_p.$ The following lemma is the analogue version of Lemma~\ref{carleson} for the case $p>q.$ 

\begin{lem}\label{Carleson2}
Let $\mu$ be a positive Borel measure on~$\overline{\D}.$ Assume that~$\mu$ is a~$(p,q)$-Carleson measure for $1\leq q < p < \infty.$ Let $0<r<1$ and~$\mu_r:=\mu_{|_{\overline{\D}\backslash r\D}}$. Then~$\mu_r$ is a~$(p,q)$-Carleson measure, and there exists a positive constant~$\gamma$ such that for every~$f\in H^p,$ \[\int_{\overline{\D}}\left|f(z)\right|^q\ \ud\mu_r(z)\leq (\|F\|_s+\gamma\|M\|_p^q\|\widetilde{G_r}\|_s)\|f\|_p^q\]
where $\ud\mu_\T=F\ud m$ and $\widetilde{G_r}(\zeta)=\int_{\Gamma(\zeta)}\frac{\ud \mu_r(z)}{1-|z|^2}.$ In addition, $\|\widetilde{G_r}\|_s\rightarrow0$ as \mbox{$r\rightarrow1.$} 
\end{lem}

We use the notation~$\widetilde{G_r}$ to avoid any confusion with the notation introduced before for~$\varphi$ and its radial function~$\varphi_r.$\\

\begin{proof}Being a~$(p,q)$-Carleson measure only depends on the ratio~$p/q$ (see \mbox{\cite[Lemma~$2.1$]{BJ}}), so we have to show that~$\mu_r$ is a~$(p/q,1)$-Carleson measure.\\
From the definition it is clear that $\widetilde{G_r}\leq G\in L^s(\T).$ Moreover~$\ud\mu_{r_{|_\T}}=\ud\mu_\T=F\ud m\in L^s(\T).$ Corollary~\ref{contcarleson} ensures the fact that~$\mu_r$ is a~$(p,q)$-Carleson measure.\\
Let $f$ be in $H^p.$ Then
\begin{align}
\int_\T\left|f(\zeta)\right|^q\ \ud\mu_r(\zeta) & =\int_\T\left|f(\zeta)\right|^q\ \ud\mu(\zeta)=\int_\T\left|f(\zeta)\right|^qF(\zeta)\ \ud m(\zeta)\nonumber\\
 & \leq\left(\int_\T\left|f(\zeta)\right|^p\ \ud m(\zeta)\right)^\frac{q}{p}\|F\|_s\nonumber\\
 & \leq\|f\|_p^q\|F\|_s\label{eq4}
\end{align}
using H\"older's inequality with conjugate exponents~$p/q$ and~$s.$\\ For $z\neq0,\ z\in\D,$ let~$\tilde{I}(z)=\{\zeta\in\T\mid\ z\in\Gamma(\zeta)\}.$ In other words~$\zeta\in\tilde{I}(z)\Leftrightarrow z\in\Gamma(\zeta).$ Then
\begin{equation}\label{eq8}
m\left(\tilde{I}(z)\right)\approx1-|z|
\end{equation}
 and 
\begin{align*}
\int_\D\left|f(z)\right|^q\ \ud\mu_r(z) &\approx \int_\D\left|f(z)\right|^q\left(\int_{\tilde{I}
(z)}\ud m(\zeta)\right)\frac{\ud\mu_r(z)}{1-|z|^2}\\
  &= \int_\T\int_{\Gamma(\zeta)}\left|f(z)\right|^q\frac{\ud\mu_r(z)}{1-|z|^2}
  \ \ud m(\zeta)\\
  &\leq \int_\T Mf(\zeta)^q\int_{\Gamma(\zeta)}\frac{\ud\mu_r(z)}{1-|z|^2}
  \ \ud m(\zeta)
\end{align*}
where $Mf(\zeta)=\sup_{z\in\Gamma(\zeta)}|f(z)|$ is the Hardy-Littlewood maximal nontangential function. We apply H\"older's inequality to obtain 
\begin{equation}\label{eq3}
\int_\D\left|f(z)\right|^q\ \ud\mu_r(z) \leq \gamma\|Mf\|_p^q\|\widetilde{G_r}\|_s\leq \gamma\|M\|_p^q\|\widetilde{G_r}\|_s\|f\|_p^q,
\end{equation} 
where $\gamma$ is a positive constant that appears in~\eqref{eq8}. Combining \eqref{eq4} and \eqref{eq3} it follows that \[\int_{\overline{\D}}\left|f(z)\right|^q\ \ud\mu_r(z)\leq (\|F\|_s+\gamma\|M\|_p^q\|\widetilde{G_r}\|_s)\|f\|_p^q.\]

 It remains to show that~$\|\widetilde{G_r}\|_s\to 0$ when~$r\to 1.$ We will make use of Lebesgue's dominated convergence theorem. Clearly we have~$0\leq \widetilde{G_r}\leq G\in L^s(\T),$ so we need to show that~$\widetilde{G_r}(\zeta)\rightarrow0$ as~$r\rightarrow1$ for~$m-$almost every~$\zeta\in\T.$ Let~\mbox{$A=\{\zeta\in\T\mid\ G(\zeta)<\infty\}$.} It is a set of full measure \mbox{($ m(A)=1$)} since $G\in L^s(\T).$ Write~$\widetilde{G_r}(\zeta)=\int_{\Gamma(\zeta)}\tilde{f_r}(z)\ \ud\mu(z)$ with~$\tilde{f_r}(z)=\ind_{\overline{\D}\backslash r\D}(z)(1-|z|^2)^{-1},\ z\in\Gamma(\zeta).$ For every $\zeta\in A$ one has 
\begin{align*}
&\left|\tilde{f_r}(z)\right|\leq\frac{1}{1-|z|^2}\in L^1\left(\Gamma(\zeta),\mu\right)\textrm{ since }\zeta\in A,\\
&\tilde{f_r}(z)             \mathop{\longrightarrow}_{r\to 1}0\textrm{ for all }z\in\Gamma(\zeta)\subset\D.
\end{align*}
Lebesgue's dominated convergence theorem in~$L^1\left(\Gamma(\zeta),\mu\right)$ ensures that $\widetilde{G_r}(\zeta)=\|\tilde{f_r}\|_{L^1\left(\Gamma(\zeta),\mu\right)}$ tends to zero as~$r$ tends to~$1$ for~$ m-$almost every~$\zeta\in\T,$ which ends the proof.
\end{proof}

\begin{thm}
Let $u$ be an analytic function on~$\D$ and $\varphi$ an analytic self-map of~$\D.$ Assume that $uC_\varphi$ is a bounded operator from~$H^p$ to~$H^q,$ with~$\infty>p>q\geq1.$ Then \[\|uC_\varphi\|_e\leq2\|C_\varphi\|_{p/q}^{1/q}\left(\int_{E_\varphi}|u(\zeta)|^\frac{pq}{p-q}\ \ud m(\zeta)\right)^\frac{p-q}{pq},\]where $\|C_\varphi\|_{p/q}$ denotes the norm of $C_\varphi$ acting on $H^{p/q}.$
\end{thm}

\begin{proof}We follow the same lines as in the proof of the upper estimate in Proposition~\ref{majoration1}: we have the decomposition~$I=K_N+R_N$ in~$B(H^p)$, where~$K_N$ is the convolution operator by the Fej\'er kernel, and \[\|uC_\varphi\|_e\leq\liminf_N\|uC_\varphi R_N\|.\]
We also have, for every $0<r<1,$

\begin{align*}
\|(uC_\varphi R_N)f\|_q^q &=\int_{\overline{\D}\backslash r\D}\vert
            R_Nf(w)\vert^q\  \ud\mu_\varphi(w)+\int_{r\D}\vert R_Nf(w)\vert^q\  \ud\mu_\varphi(w) \\
            &=I_1(N,r,f)+I_2(N,r,f).
\end{align*}
As in the $p\leq q$ case, we show that \[\lim_N\sup_{\|f\|_p\leq1}I_2(N,r,f)=0.\]
The measure $\mu_\varphi$ being a~$(p,q)$-Carleson measure, we use Lemma~\ref{Carleson2} to have the following inequality \[I_1(N,r,f)\leq (\|F\|_s+\gamma\|M\|_p^q\|\widetilde{G_r}\|_s)\|R_Nf\|_p^q\]for every $f\in H^p.$ As a consequence\[\|uC_\varphi\|_e\leq\liminf_N\left(\sup_{\|f\|_p\leq1}I_1(N,r,f)\right)^\frac{1}{q}\leq 2(\|F\|_s+\gamma\|M\|_p^q\|\widetilde{G_r}\|_s)^\frac{1}{q}\]
using the fact that~$\sup_N\|R_N\|\leq2.$ Now we make~$r$ tend to~$1$, keeping in mind that~$\|\widetilde{G_r}\|_s\to0.$ We obtain \[\|uC_\varphi\|_e\leq 2\|F\|_s^{1/q}.\]
It remains to see that we can choose~$F$ in such a way that
 \[\|F\|_s\leq\|C_\varphi\|_{p/q}\left(\int_{E_\varphi}|u(\zeta)|^\frac{pq}{p-q}\  \ud m(\zeta)\right)^{1/s}.\]
Indeed, if $f\in C(\T)\cap H^{p/q}$, we apply H\"older's inequality with conjugates exponents~$p/q$ and~$s$ to have 
\begin{align*}
\left|\int_\T f \  \ud\mu_{\varphi,\T}\right| & =\left|\int_{E_\varphi}|u|^q f\circ\varphi\  \ud m\right|\\
 & \leq\int_{E_\varphi}|u|^q|f\circ\varphi|\  \ud m\\
 & \leq\|C_\varphi(f)\|_{p/q}\left(\int_{E_\varphi}|u|^{sq}\  \ud m\right)^{1/s},
\end{align*} 
meaning that~$\mu_{\varphi,\T}\in\left(H^{p/q}\right)^*$, which is isometrically isomorphic to~$L^s(\T)/H^s_0$, where~$H^s_0$ is the subspace of~$H^s$ consisting of functions vanishing at zero. If we denote by~$N(\mu_{\varphi,\T})$ the norm of~$\mu_{\varphi,\T}$ viewed as an element of~$\left(H^{p/q}\right)^*$, then one can choose~$F\in L^s(\T)$ satisfying\[\|F\|_s=N(\mu_{\varphi,\T})\leq\|C_\varphi\|_{p/q}\left(\int_{E_\varphi}|u|^{pq/(p-q)}\  \ud m\right)^{1/s}\] and~$\mu_{\varphi,\T}=F\  \ud m$ (see for instance~\cite{K}, p. $194$). Finally we have
 \[\|uC_\varphi\|_e\leq2\|C_\varphi\|_{p/q}^{1/q}\left(\int_{E_\varphi}|u(\zeta)|^\frac{pq}{p-q}\ \ud m(\zeta)\right)^\frac{p-q}{pq}.\]
\end{proof}

Although we have not be able to give a corresponding lower bound of this form for the essential norm of~$uC_\varphi$, we have the following result:

\begin{prop}
Let $1\leq q<p<\infty$, and assume that $uC_\varphi\in B(H^p,H^q).$ Then 
\[\|uC_\varphi\|_e\geq\left(\int_{E_\varphi}|u(\zeta)|^q\ \ud m(\zeta)\right)^\frac{1}{q}.\]
\end{prop}

\begin{proof}Take a compact operator $K$ from~$H^p$ to~$H^q$. Since it is completely continuous, and the sequence $(z^n)$ converges weakly to zero in~$H^p,$ $\left(K(z^n)\right)_n$ converges to zero in~$H^q.$ Hence\[\|uC_\varphi+K\|\geq\|(uC_\varphi+K)z^n\|_q\geq\|uC_\varphi(z^n)\|_q-\|K(z^n)\|_q\]for every~$n\geq0.$ Taking the limit as~$n$ tends to infinity, we have\[\|uC_\varphi\|_e\geq\left(\int_{E_\varphi}|u(\zeta)|^q\ \ud m(\zeta)\right)^\frac{1}{q}.\]
\end{proof}

\subsection*{Acknowledgements}
The author is grateful to the referee for his careful reading and for the several suggestions made for improvement.

\end{document}